\documentclass[reqno,11pt]{amsart}

\usepackage{graphicx} 
\usepackage{amsmath,amssymb,amsthm,tikz}
\usepackage{xcolor}
\usepackage{esint}
\usepackage{booktabs}
\usepackage{scalerel}
\usepackage[usestackEOL]{stackengine}
\usepackage{mathrsfs}
\usepackage[normalem]{ulem}
\usepackage[colorinlistoftodos,prependcaption,textsize=tiny]{todonotes}
\usepackage{parskip}
\def\dashint{\,\ThisStyle{\ensurestackMath{
            \stackinset{c}{.2\LMpt}{c}{.5\LMpt}{\SavedStyle-}{\SavedStyle\phantom{\int}}}
        \setbox0=\hbox{$\SavedStyle\int\,$}\kern-\wd0}\int}

\def \N {\mathbb{N}} 
\def \R {\mathbb{R}}
\def \dist {\mathrm{dist}}

\def \suchthat {~\big{|} ~}

\DeclareMathOperator*{\ess}{ess\,}

\newtheorem{theorem}{Theorem} 
\newtheorem{lemma}{Lemma}
\newtheorem{proposition}{Proposition}

\newtheorem{remark}{Remark}

\numberwithin{equation}{section}

\title[Oscillatory FBP in heterogeneous media]{Oscillatory free boundary problems in heterogeneous media}

\author[D.J. Ara\'ujo]{Dami\~ao J. Ara\'ujo}
\address{Department of Mathematics, Universidade Federal da Para\'iba, 58059-900, Jo\~ao Pessoa-PB, Brazil}{}
\email{araujo@mat.ufpb.br}

\author[G.S. Sá]{Ginaldo S. Sá}
\address{Centro de Modelamiento Matemático (CNRS IRL2807), Universidad de Chile, Santiago, Chile}{}
\email{gdesantana@dim.uchile.cl}

\author[E. V. Teixeira]{Eduardo V. Teixeira}
\address{Department of Mathematics, Oklahoma State University, 401 Mathematical
Sciences, Stillwater, OK 74078, USA}
\email{eduardo.teixeira@okstate.edu}

\author[J.M. Urbano]{Jos\'{e} Miguel Urbano}
\address{Applied Mathematics and Computational Sciences (AMCS), Computer, Electrical and Mathematical Sciences and Engineering Division (CEMSE), King Abdullah University of Science and Technology (KAUST), Thuwal, 23955-6900, Kingdom of Saudi Arabia; and CMUC, Department of Mathematics, University of Coimbra, 3000-143 Coimbra, Portugal}
\email{miguel.urbano@kaust.edu.sa}

\begin{document}



\begin{abstract}
We investigate a class of free boundary problems with oscillatory singularities within heterogeneous media. Our main result yields sharp regularity estimates along the free boundary, provided the power of the singularity varies in a Dini-continuous fashion below a certain threshold. We also reveal an interesting repelling estimate preventing the free boundary from touching the region where the singularity power oscillates above the threshold.
\end{abstract}  

\maketitle

\section{Introduction}

In this work, we investigate non-negative local minimizers of energy functionals with varying singularities, taking into account the complexities arising from the heterogeneous nature of the media. More precisely, we are interested in regularity estimates for local minimizers of the energy functional
\begin{equation}\label{funct}
\mathscr{F}(u) := \int_{\Omega}  \langle A(x) \nabla u, \nabla u \rangle + \lambda(x)  u^{\gamma(x)} \chi_{\{u> 0 \}}  \,dx,
\end{equation}
where $\Omega \subset \mathbb{R}^n$, $n\ge 3$, is an open bounded set. The matrix $A(x)$ is only assumed to be symmetric, measurable, and uniformly elliptic, \textit{i.e.}, 
\begin{equation}\label{ellipticity_condition}
    \mu |\xi|^2\leq \langle A(x)\xi , \xi \rangle \leq \mu^{-1}|\xi|^2,
\end{equation}
for some positive constant $0< \mu \le 1$. The forcing term $\lambda(x)$ is bounded away from zero and infinity, \textit{i.e.}, $0 < c_0\le \lambda (x) \le \Lambda$. The exponent function $\gamma(x)$ is non-negative, measurable, with 
\begin{equation}\label{gamma bounds}
    0\le \gamma(x) \le \gamma^{\star} < 2^{\ast}, 
\end{equation}
where $\gamma^\star:=\ess\sup_{x\in \Omega} \gamma(x)$ and $2^{\ast} = \frac{2n}{n-2}$ stands for the critical Sobolev exponent. 

Three essential features of the model considered in this paper should be noted. Firstly, it is a free boundary problem which allows for a continuum range of singularities, \textit{i.e.}, the exponent $\gamma$ varies point-by-point. The static version, when $0< \gamma(x) \equiv \gamma_0<1$, boasts a notable historical foundation, tracing its origins to the pioneering contributions of Phillips \cite{P1, P2} and Alt--Phillips \cite{AP}. These works have inspired substantial advances in the theory of free boundary problems throughout the years; see, for instance, \cite{AT, DSS, DSS1, DSS2, ES, DKV, ST, Teix, WY, Y}, among others. 

The analysis of free boundary problems shaped by oscillatory singularities offers a robust mathematical framework for modeling more realistic phenomena, but it is considerably more involved. The inaugural work developed in \cite{ASTU1}, where additional insights on the importance of such models can be found, provides a thorough analysis of the problem wherein diffusion is governed by the Laplace operator, \textit{i.e.}, assuming the matrix coefficient to be the identity, $A(x) = \text{Id}$. This sets the stage for the second and primary innovation addressed in the present paper: an extension of the modeling framework to include heterogeneous media. We thus include materials characterized by randomness or variability in their properties, introducing an element of unpredictability. Hence, the mathematical setup leads to the analysis of such problems with no predetermined structure, and consequently, we consider the coefficient matrix $A(x)$ to be merely measurable and elliptic. Notably, in the realm of stochastic materials, the formulation of variational free boundary problems proves to be exceptionally challenging, and the current state of our understanding regarding solutions and their corresponding free boundaries in this context remains somewhat limited.  

The third new feature of the model treated in this paper concerns the fact that we broaden the range in which the singularity varies, allowing $\gamma(x)$ to oscillate between $0$ and  $2^{\ast}$. The analysis in \cite{ASTU1} is limited to singularities oscillating between $0$ and $1$. While this may initially seem a minor extension, it allowed us to discover an interesting new feature of free boundary problems shaped via oscillatory singularities: a sort of free boundary repelling estimate, as described in the last part of our main theorem. We further note that the subcritical bound for the oscillation range will be used below for the more general class of signed potentials and is required for the boundedness of local minimizers. This is due to considerations of conformal invariance, as in the classical Yamabe problem. 

All the results are local, so we consider, without loss of generality, the simplest case, $\Omega=B_1$. The main result of this paper is as follows.

\begin{theorem}\label{thm-main}
Let $u$ be a non-negative local minimizer of the energy functional \eqref{funct} in $B_1\subset \mathbb{R}^n$. Then 

\begin{enumerate}

\item $u \in C^{0,\epsilon}(B_{1/2})$, for some universal $0<\epsilon = \epsilon(n,\mu, \gamma^{\star})< 1$, and
$$
    \|u\|_{C^{0,\epsilon}(B_{1/2})} \le C\left(1+\|u\|_{L^2(B_1)}+\|u\|_{L^2(B_1)}^{\gamma^\star/2}\right),
$$
for a constant $C$, depending only on $n, \mu, \gamma^{\star}, \Lambda$. 

\medskip

\item If $\gamma(x)$ is Dini continuous at $x_0 \in B_{1/2} \cap \partial \{u > 0 \}$ and $\gamma(x_0) < 2$, then
\begin{equation}\label{sharp growth}
    \sup\limits_{B_r(x_0)} u \le C_0 \, r^{\frac{2}{2-\gamma(x_0)}},
\end{equation}
for $0<r\le 1/4$ and a constant $C_0$, depending only on $n, \mu, \gamma^{\star}, \Lambda$, $\|u\|_{L^\infty(B_{3/4})}$, $2-\gamma(x_0)$, and the Dini modulus of continuity of $\gamma$ at $x_0$. 

\medskip

\item Finally, if $\gamma(x)$ is continuous, with modulus of continuity $\omega$,  and $x_0 \in B_{1/2}$ is such that $\gamma(x_0) > 2$, then there exists a constant $\varrho_0>0$, depending only on $n, \mu, \gamma^{\star}, \Lambda$,   $\|u\|_\infty$, $\gamma(x_0)-2$, and $\omega$, such that
\begin{equation}\label{eq thm repel}
    \dist \left( x_0, \partial \{u > 0\} \right) \ge \varrho_0.
\end{equation}
In particular, if $x_0 \in Z(u) := \{u=0\} \cap B_{1/2}$, then $x_0$ must be an interior point of the coincidence set $Z(u)$.
\end{enumerate}

\end{theorem}

Probably, the primary contribution of Theorem \ref{thm-main} concerns the growth estimate \eqref{sharp growth}, which yields a considerable gain of regularity along the associated free boundaries. As noted in the first part of Theorem \ref{thm-main}, local minimizers are only $C^{0,\epsilon}$ within the non-coincidence set $\{u> 0 \}$. A comparable regularity estimate was previously established in \cite{ASTU1} for the case where $A(x) \equiv \text{Id}$. In that context, however, local minimizers are proven to be of class $C^{1,\alpha}$ everywhere. The sharp regularity estimate is then derived from such local regularity properties of minimizers, which are subsequently transferred to the free boundary points through the continuity properties of $\gamma$. The analysis to obtain the corresponding estimate \eqref{sharp growth} in that case was considerably softer than that carried out in this paper. The sharpness of estimate \eqref{sharp growth} follows directly from \cite[Theorem 4.2]{ASTU1}.

On the other hand, estimate \eqref{eq thm repel} should be interpreted as a sort of {\it unique continuation} property tailored explicitly for this class of free boundary models. Heuristically, it asserts that if $u$ possesses an infinite-order zero, then $u$ must be identically zero within a quantifiably defined ball. It is worth noting that, when $A(x)$ is Lipschitz continuous, such a result can be obtained via an Almgren-type monotonicity formula applied to the associated Euler-Lagrange equation. In contrast, in the setting of a measurable medium considered in this paper, no such monotonicity tool is available, and new arguments are therefore required. The proof unfolds through an arguably charming blow-up analysis, reminiscent of the approach used in \cite{CKS}, albeit with a distinctive emphasis and nature, combined with Harnack's inequality for the Euler--Lagrange equation, whose zeroth-order coefficient is bounded when $\gamma\ge2$. We further comment that the proof delivered in Section \ref{sct repel} has additional implications. Actually, estimate \eqref{eq thm repel} remains valid under the condition that $\gamma(x) \ge 2$ within a specified ball. In essence, our proof establishes that the free boundary stays away from the interior of the set $\gamma^{-1}[2,2^*)$. 

We conclude this introduction by pointing out a natural extension of the present framework. From a variational perspective, it is meaningful to consider energy functionals whose singular exponent is allowed to vary in a wider range, including negative values, namely $\gamma \in (-2,2^{*})$. In this regime, local minimizers are expected to be only H\"older continuous at free boundary points corresponding to negative exponents, reflecting a loss of regularity driven by the more singular nature of the energy. This setting is particularly appealing given its potential connections to geometric variational problems. In fact, as suggested by the analysis developed in the recent paper \cite{STU}, such models may serve as effective analytic approximations to anisotropic minimal surface energies. We believe that extending the present theory to encompass negative exponents and anisotropic structures could open new avenues of research at the interface of free boundary problems, geometric measure theory, and singular variational models, and we hope this work will stimulate further investigations in this direction.

The paper is structured as follows. In Section \ref{sct Prelim}, we outline some elementary, though fundamental, results concerning local minimizers of the energy functional investigated in this study. This section includes a proof of local boundedness and universal H\"older regularity. Moving to Section \ref{sct reg gain}, we establish the regularity gain along the free boundary. Lastly, in Section \ref{sct repel}, we establish the third and final piece of information encoded in Theorem \ref{thm-main}: a quantified repelling estimate that prevents free boundaries from intersecting the set $\gamma^{-1}(2, 2^{*})$. 

\section{Preliminary results} \label{sct Prelim}

In this section, we discuss some elementary results about local minimizers of the energy functional \eqref{funct}. Additionally, we explore essential concepts related to the problem that, while somewhat well-known, are introduced here for the convenience of the readers. 

For each subset $\mathcal{O} \subset \Omega$, and given an energy functional $F(\mathcal{O}, v)$ acting on the Sobolev space $H^1(\Omega)$, a function $w\in H^1(\Omega)$ is said to be a local minimizer of $F$ if 
$$
    F(\mathcal{O}, w)\leq F (\mathcal{O}, v),
$$
for all $v$ such that $v-w \in H^1_0(\mathcal{O})$. The regularity theorems we will prove in this paper concern local minimizers of the energy functional \eqref{funct}. This is an abundant class of functions. Indeed, we start by commenting that given a boundary datum $\phi \in H^1(\Omega)$, there always exists a global minimizer (and thus a local minimizer) of the energy functional \eqref{funct} taking $\phi$ as a boundary value in the sense of the traces. The proof follows standard procedures, and we provide a brief sketch below. For notational convenience, given $\phi \in H^1(\Omega)$, let us denote the affine space 
$$
    \mathcal{K}_\phi := \left \{ v\in H^1(\Omega) \suchthat v-\phi \in H^1_0(\Omega) \right \}.
$$
\begin{proposition}\label{existence_solution}
Given $\phi \in H^1(\Omega)$, there exists $u\in \mathcal{K}_\phi$ satisfying
$$
    \mathscr{F}(u) = \min\limits_{v\in \mathcal{K}_\phi} \mathscr{F}(v).
$$
Furthermore, $u$ is non-negative provided $ \phi \ge 0$. Also, if $\lambda \ge 0$ and $\phi \in L^\infty$, then $\|u\|_\infty \le \|\phi\|_\infty$.
\end{proposition}

\begin{proof}
Set $m:=\inf_{v\in \mathcal{K}_\phi}\mathscr{F}(v)$ and let $(v_k)_{k\ge1} \in \mathcal{K}_\phi$ be a minimizing sequence. By ellipticity and Poincar\'e's inequality, the sequence $(v_k)_{k\ge1}$ is bounded in $H^1$, thus, up to a subsequence, $v_k\rightarrow u$ weakly in $H^1$ and, due to Sobolev embedding, strongly in $L^2(\Omega)$ and almost everywhere. Clearly $u \in \mathcal{K}_\phi$. By weak lower semi-continuity and Fatou's lemma, we conclude 
$$
    m \le \mathscr{F}(u) \le \liminf\limits_{k\to \infty} \mathscr{F}(v_k) = m,
$$
and thus $u$ is a minimizer of $\mathscr{F}$ over $\mathcal{K}_\phi$. 

Next, if $\phi \ge 0$, $u^{+}$ competes with $u$ in the minimization problem. We then deduce, by ellipticity,
$$
    \int_{\{u < 0 \}} |\nabla u |^2 dx \le 0.
$$
Finally, if $\phi \in L^\infty$, with $M:=\|\phi\|_\infty$, and $\lambda \ge 0$, the function $u_M :=\min\{u,M\}$ competes with $u$ and we similarly deduce 
$$
    \mu \int_{\{u > M \}} |\nabla u |^2 dx \le \int_{\{u > M \}} \lambda(x) \left[ M^{\gamma(x)} - u^{\gamma(x)}\right] dx \le 0.
$$
\end{proof}

Next, we record the scaling feature of the energy functional $\mathscr{F}$ in a lemma. The proof follows from direct calculation via a change of variables.

\begin{lemma}[Scaling] \label{scaling}
Let $u\in H^1(\Omega)$ be a local minimizer of \eqref{funct}. Given $x_0 \in \Omega$,  $0<\varrho < \dist(x_0, \partial \Omega)$ and $\sigma>0$, the function
$$
    v (x) := \frac{u(x_0 + \varrho x)}{\sigma}, \quad x\in B_1
$$
is a local minimizer of the functional
$$
    \mathscr{F}_{x_0, \varrho, \sigma}(v) := \int_{B_1} \left( \langle A_{x_0, \varrho}(x)\nabla v, \nabla v \rangle  + \lambda_{x_0, \varrho, \sigma}(x) v^{\gamma_{x_0,\varrho} (x)} \chi_{\{v> 0 \}} \right) dx,
$$
where
$$
    \begin{array}{rll}
         A_{x_0, \varrho}(x) &:=& A(x_0+\varrho x), \\
         \lambda_{x_0, \varrho, \sigma}(x) &:=& \sigma^{\gamma(x_0+\varrho x)}\left(\frac{\varrho}{\sigma} \right)^2  \lambda(x_0 + \varrho x)\\  
          \gamma_{x_0, \varrho}(x) &:=& \gamma(x_0 + \varrho x).
    \end{array}
$$
\end{lemma} 

\medskip

\subsection{$L^\infty$ bounds} \label{subsct Linfty}

In what follows, we will consider a more general family of energy functionals of the form
\begin{equation}\label{functional_rho}
    \mathscr{F}_{\rho}(u):=\int_\Omega  \langle A(x)\nabla u, \nabla u \rangle + \rho(x,u) \, dx,
\end{equation}
where $\rho \colon \Omega\times \mathbb{R} \rightarrow\mathbb{R}$ satisfies  
\begin{equation}\label{bounded_pho}
    \left| \rho(x,u) \right| \leq f(x) +M|u|^\gamma,
\end{equation}
for a non-negative function $f\in L^q$, with $q>\frac{n}{2}$, and $2\le \gamma<2^\ast=\frac{2n}{n-2}$.

It is worth remarking that we do not require the differentiability of $\rho$ with respect to its $u$ argument. We assume subcritical growth in the general functional considered here. Finally, we note that the functional \eqref{funct} trivially satisfies \eqref{bounded_pho}, with $\gamma=\max\{2,\gamma^\star\}$, since 
$$
    \left | \lambda (x) u^{\gamma(x)} \chi_{\{u>0\}} \right | \le \|\lambda\|_\infty \max\{1, |u| \}^{\gamma^{\star}} \le \|\lambda\|_\infty (1 + |u|^{\gamma^{\star}}).
$$

We next comment on an $L^\infty-$bound for local minimizers of $\mathscr{F}_{\rho}$. The rationale follows a somewhat standard path along the lines of \cite{GG}; yet, we opt to include a proof as a courtesy to the readers.

\begin{proposition}\label{blocal}
Let $u$ be a local minimizer of \eqref{functional_rho} in $B_1$ under condition \eqref{bounded_pho}. Then, 
$$
\|u\|_{L^\infty(B_{1/2})} \le C\left(\|u\|_{L^2(B_1)}+\|f\|_{L^q(B_1)}^{1/2}\right),
$$
for a constant $C$ depending only on $n,\mu,q,M,\gamma$ and, when $\gamma>2$, on $\|u\|_{L^{2^\ast}(B_1)}$.
\end{proposition}
\begin{proof}
The proof is based on De Giorgi's method. For $0<t<1$ and $k \geq 0$, denote
$$
    \Omega^k(t) := \{x\in B_t \suchthat u-k> 0 \}.
$$
For $0<r< R \ll 1$, to be set \textit{a posteriori}, let $\eta$ be a radially symmetric cut-off function $\eta \in C_0^{\infty}(B_R)$, with $0\le \eta \le 1$, $\eta \equiv 1$ in $B_r$, and $|D\eta| \le 2(R-r)^{-1}$. Denote $w_k = (u-k)^{+}$ and consider the competing function $v:= u - \eta w_k$. Since $u$ is a local minimizer of $\mathscr{F}_{\rho}$, we have
$$
    \mathscr{F}_{\rho} \left( \Omega^k(R), u \right)  \le \mathscr{F}_{\rho} \left( \Omega^k(R), v \right),
$$
which leads to
\begin{equation}\label{Linfty Eq1}
    \int_{\Omega^k(R)} \left|\nabla u\right|^2 dx \le C_1 \left( \int_{\Omega^k(R)} \left|\nabla v\right|^2 dx +  \int_{\Omega^k(R)} f(x) + \left|u\right|^\gamma + \left|v\right|^\gamma dx \right).
\end{equation}
Note that, in $\Omega^k(R)$, we have
$$
    u = (1-\eta)u + \eta(w_k + k) \quad \text{ and } \quad v = (1-\eta)u + k \eta.
$$
Hence
$$
   |u|^\gamma + |v|^\gamma \le C_2 \left( (\eta w_k)^\gamma + |u|^\gamma(1-\eta)^\gamma  + \eta^\gamma k^\gamma \right)
$$
and
\begin{equation}\label{Linfty Eq5}
   |\nabla v|^2  \le C_2 \left ( (1-\eta)^2 |\nabla u|^2 + \frac{1}{(R-r)^2} (u-k)^2 \right ).
\end{equation}
By H\"older's inequality, we estimate
\begin{eqnarray*} 
    \int_{\Omega^k(R)} (\eta w_k)^\gamma dx &=& \displaystyle  \int_{\Omega^k(R)} (\eta w_k)^{2}  (\eta w_k)^{\gamma - 2} dx \\
    &\le& \displaystyle  \left ( \int_{\Omega^k(R)} (\eta w_k)^{2^\ast} dx \right)^{\frac{2}{2^\ast}} \left ( \int_{B_R} (\eta w_k)^{\frac{(\gamma - 2)n}{2}} dx \right)^{\frac{2}{n}}.
\end{eqnarray*}
Note that the quantity
$$
    \zeta(s) := \left ( \int_{B_s} (\eta w_k)^{\frac{(\gamma - 2)n}{2}} dx \right)^{2/n}
$$
is a modulus of continuity, uniform in $k$. Indeed, by H\"older's inequality, we can estimate
$$
    \zeta(s) \le \left ( \int_{B_s} |u|^{\frac{(\gamma - 2)n}{2}} dx \right)^{2/n} \le \|u\|^{\gamma - 2}_{L^{2^{\ast}}} |B_s|^{1-\frac{\gamma}{2^{\ast}}}.
$$
Note that we have $\gamma < 2^{\ast}$. Applying the Sobolev inequality to $\eta w_k$, we obtain
\begin{eqnarray*}
    \int_{\Omega^k(R)} (\eta w_k)^\gamma dx &\le& \displaystyle S\zeta(R) \left ( \int_{\Omega^k(R)} |\nabla (\eta w_k)|^2 dx \right ) \\
    &\le & \displaystyle S C_2 \zeta(R) \left ( \int_{\Omega^k(R)} |\nabla u|^2 dx + \frac{1}{(R-r)^2} \int_{\Omega^k(R)}  w_k^2 dx \right ),
\end{eqnarray*}
where $S$ is the Sobolev constant. We can add to both sides in \eqref{Linfty Eq1} the term
$$ 
    \int_{\Omega^k(R)} \left| u \right|^\gamma dx.
$$ 
Now, choose $0< R \ll 1$ small enough so that $C_*S\zeta(R)\le 1/2$, where $C_*$ depends only on $n,\mu,M,\gamma$. Combining all estimates above, keeping in mind $\eta \equiv 1$ in $B_r$, yields 
\begin{align*}
 \int_{\Omega^k(R)} (|\nabla u|^2+|u|^\gamma)\,dx
 &\le C\int_{\Omega^k(R)\setminus\Omega^k(r)}(|\nabla u|^2+|u|^\gamma)\,dx\\
 &\quad+\frac{C}{(R-r)^2}\int_{\Omega^k(R)}w_k^2\,dx\\
 &\quad+Ck^\gamma|\Omega^k(R)|+C\int_{\Omega^k(R)}f\,dx.
\end{align*}

Applying a standard Analysis lemma, similar to \cite[Lemma 1.1]{GG}, we reach
\begin{align*}
 \int_{\Omega^k(r)} (|\nabla u|^2+|u|^\gamma)\,dx
 &\le \frac{C}{(R-r)^2}\int_{\Omega^k(R)} w_k^2\,dx\\
 &\quad+C k^\gamma|\Omega^k(R)|+C\int_{\Omega^k(R)}f\,dx,
\end{align*}
which readily yields
\begin{equation}\label{Linfty Eq10}
\begin{aligned}
 \int_{\Omega^k(r)}|\nabla u|^2\,dx
 &\le \frac{C}{(R-r)^2}\int_{\Omega^k(R)}(u-k)^2\,dx\\
 &\quad+C k^\gamma|\Omega^k(R)|
       +C\|f\|_{L^q(B_R)}|\Omega^k(R)|^{1-1/q}.
\end{aligned}
\end{equation}
To include the forcing term in the iteration, put $P=\|u\|_{L^{2^\ast}(B_1)}$ and $s=2^\ast/(\gamma-2)>n/2$ if $\gamma>2$. Then
\[
 k^\gamma|\Omega^k(R)|\le P^{\gamma-2}k^2
             |\Omega^k(R)|^{1-1/s}.
\]
For $\gamma=2$, take $s=\infty$ and replace $P^{\gamma-2}$ by $1$.
Set $t=\min\{q,s\}$ and $\theta=2/n-1/t>0$.
On a ball of a fixed radius $R$ satisfying the absorption condition above, use
\[
 R_j=R/2+R/2^{j+1},\qquad k_j=D(1-2^{-j-1}),\qquad
 Y_j=\int_{B_{R_j}}(u-k_j)_+^2\,dx.
\]
For $D\ge 2\|f\|_{L^q(B_1)}^{1/2}$, Sobolev's inequality and \eqref{Linfty Eq10}, applied at level $k_{j+1}$ on intermediate concentric balls, give, with $Z_j=Y_j/D^2$,
\[
 Z_{j+1}\le Cb^j\bigl(Z_j^{1+2/n}+Z_j^{1+\theta}\bigr),\qquad b>1.
\]
Here one uses $|\Omega^{k_{j+1}}(R_j)|\le
 (k_{j+1}-k_j)^{-2}Y_j$; the constants depend only on the data listed in the statement and on the chosen $R$.
Taking $D=C(\|u\|_{L^2(B_1)}+\|f\|_{L^q(B_1)}^{1/2})$ with $C$ large enough makes $Z_0$ small, so the iteration yields $Z_j\to0$ and $u\le D$ on $B_{R/2}$. The case $D=0$ is immediate. Repeating the argument for $-u$ and covering $B_{1/2}$ by such balls proves the proposition.
\end{proof}

For the original functional \eqref{funct}, a sharper bound holds. Since $u$ is non-negative and $A$-subharmonic, the standard local estimate \cite[Theorem 8.17]{GT} yields
\begin{equation}\label{original-bound}
 \|u\|_{L^\infty(B_{3/4})}\le C(n,\mu)\|u\|_{L^2(B_1)}.
\end{equation}

\subsection{H\"older continuity} \label{subsct Holder}
Owing to inequality \eqref{Linfty Eq10}, combined with the fact that if $u$ is a local minimizer of 
$\mathscr{F}_{\rho(x,u)}$ then $-u$ is a local minimizer of $\mathscr{F}_{\rho(x,-u)}$, one can eventually proceed with the classical De Giorgi's proof, as delineated in \cite{DG} and in \cite{GG}, to attain interior H\"older continuity of local minimizers. We shall adopt a different strategy inspired by the recent papers \cite{LST, PTeix, STeix}. In addition to its arguably more appealing geometric flavor, this approach provides the advantage of introducing fundamental tools that will be utilized in the subsequent sections.

By Proposition \ref{blocal}, $u$ is bounded on every interior ball. On such a ball let $K=\|u\|_\infty$, set $T_K(t)=\max\{-K,\min\{t,K\}\}$, and replace $\rho(x,t)$ by $\rho(x,T_K(t))$. Truncating competitors shows that $u$ remains a local minimizer, while the new potential is bounded by $f(x)+MK^\gamma$. After rescaling the ball and relabelling this bound as $f$, we may therefore assume, throughout the following three lemmas, that
\begin{equation}\label{bounded-perturbation}
 |\rho(x,t)|\le f(x),\qquad f\in L^q(B_1),\quad q>n/2.
\end{equation}

Our first result concerns a local minimizer's proximity (properly measured) to its $A$-harmonic replacement, namely, the unique function $h$ in $B_r$ satisfying $h=u$ on $\partial B_r$ and
$$
    \int_{B_r} \langle A(x)\nabla h , \nabla w \rangle \, dx= 0,
$$
for all $w\in H^1_0(B_r)$.

\begin{lemma}\label{Approx_lemma}
Let $u$ be a local minimizer of \eqref{functional_rho}, normalized so as to satisfy 
\begin{equation} \label{hypo}
\dashint_{B_1} |u(x)|^2 dx\leq 1,
\end{equation}
and $h$ be the $A$-harmonic replacement of $u$ in $B_{1/2}$. There exists a constant $C=C(n, \mu)>1$ such that, 
\begin{equation}
\dashint_{B_{1/2}} |u(x)-h(x)|^2\,dx\le C\|f\|_{L^q(B_{1/2})},
\end{equation}
where $f$ is the bound prescribed in \eqref{bounded-perturbation}. 
\end{lemma}

\begin{proof}
Since $h$ competes with $u$ in the minimization problem, we readily have 
$$
    \int_{B_{1/2}}\left(  \langle A(x)\nabla u , \nabla u \rangle -  \langle A(x)\nabla h , \nabla h \rangle \right) dx \leq \int_{B_{1/2}} \left( \rho(x,h)-\rho(x,u)\right) dx.
$$
Now, since $h$ is the $A$-harmonic replacement of $u$, we have
$$
    \int_{B_{1/2}} \langle A(x) \nabla h , \nabla h \rangle \, dx = \int_{B_{1/2}} \langle A(x) \nabla h , \nabla u \rangle \, dx,
$$
and, due to the symmetry and the ellipticity of $A$, this implies that
\begin{align*}
 &\int_{B_{1/2}}\bigl(\langle A(x)\nabla u,\nabla u\rangle
                 -\langle A(x)\nabla h,\nabla h\rangle\bigr)\,dx\\
 &\qquad=\int_{B_{1/2}}\langle A(x)\nabla(u-h),\nabla(u-h)\rangle\,dx\\
 &\qquad\ge\mu\int_{B_{1/2}}|\nabla(u-h)|^2\,dx.
\end{align*}

Using \eqref{bounded-perturbation} and H\"older's inequality, with $B=B_{1/2}$, we obtain
\begin{equation}\label{use later}
\begin{aligned}
 \int_B\bigl(\rho(x,h)-\rho(x,u)\bigr)\,dx
 &\le\int_B\bigl(|\rho(x,h)|+|\rho(x,u)|\bigr)\,dx\\
 &\le 2\int_B f\,dx
 \le 2|B|^{1-1/q}\|f\|_{L^q(B)}.
\end{aligned}
\end{equation}

Combining the above estimates yields
$$
    \dashint_{B_{1/2}} |\nabla(u-h)|^2\,dx\le\frac{C}{\mu}\|f\|_{L^q(B_{1/2})}, 
$$
and, using the Poincar\'e inequality in balls, we conclude 
$$
    \dashint_{B_{1/2}} |u-h|^2\,dx\le C\|f\|_{L^q(B_{1/2})},
$$
as desired.
\end{proof}

We next explore this proximity to the $A$-harmonic replacement and the regularity properties of the latter to start  preparing the proof of the H\"older continuity of local minimizers. Recall that, by De Giorgi--Nash--Moser regularity theory, $h$ is locally H\"older continuous, satisfying, for an (optimal) exponent $\alpha_0\in (0,1)$, the estimate
\begin{equation}\label{A-expoent}
    \left\| h \right\|_{C^{0,\alpha_0}(B_{1/4})} \leq C  \left\| h \right\|_{L^2(B_{1/2})}.
\end{equation}

\begin{lemma}\label{int_u}
Let $u$ be a local minimizer of \eqref{functional_rho} satisfying \eqref{hypo}. For any 
\begin{equation} \label{agniek}
    0<\alpha< \min \left\{ 1-\frac{n}{2q} , \alpha_0 \right\},
\end{equation}
where $\alpha_0$ is from \eqref{A-expoent}, there exists $\tau_0 \in \left( 0, \frac14 \right)$, $\delta_0>0$ and a uniformly bounded constant $b$ such that, if the bound in \eqref{bounded-perturbation} satisfies
\begin{equation}\label{chochin}
    \|f\|_{L^q(B_1)}\le\delta_0,
\end{equation}
then
$$
    \dashint_{B_{\tau_0}} \left| u(x)-b \right|^2 dx \leq \tau_0^{2\alpha}.
$$
\end{lemma}

\begin{proof}
With $\tau_0 \in \left( 0, \frac14 \right)$ and $0<\delta_0\le1$ to be fixed, apply Lemma \ref{Approx_lemma}, to get
$$
    \dashint_{B_{\tau_0}}|u(x)-h(x)|^2\,dx\le C\tau_0^{-n}\|f\|_{L^q(B_{1/2})}\le C\tau_0^{-n}\delta_0,
$$
where $h$ is the $A$-harmonic replacement of $u$ in $B_{1/2}$. By local H\"older regularity, there exists $C_0>0$ such that
$$
    |h(x)-h(0)|\leq C_0|x|^{\alpha_0},
$$
for any $x\in B_{\tau_0}$. Then,
\begin{eqnarray*}
\dashint_{B_{\tau_0}} \left| u(x)-h(0) \right|^2 dx &\leq & 2\dashint_{B_{\tau_0}} \left| u(x)-h(x) \right|^2 dx\\ 
& & + 2\dashint_{B_{\tau_0}} \left| h(x)-h(0) \right|^2 dx\\
&\leq & 2C \tau_0^{-n} \delta_0 +2C_0^2 \tau_0^{2\alpha_0}.
\end{eqnarray*}
Now we make our selections: first, we choose $\tau_0 \in \left( 0, \frac14 \right)$ so small that 
$$
    2C_0^2 \tau_0^{2\alpha_0} \leq \frac{\tau_0^{2\alpha}}{2},
$$
which is possible since $\alpha_0-\alpha >0$, and then, we select 
$$
   \delta_0 = \frac{\tau_0^{2\alpha +n}}{4C}.
$$ 
The proof is complete.
\end{proof}

\begin{remark}
The normalization \eqref{hypo}, Lemma \ref{Approx_lemma} and \eqref{chochin} give a uniform $L^2(B_{1/2})$ bound for $h$. Interior harmonic estimates therefore bound $b=h(0)$ uniformly. This means that there exists a constant $L>0$, depending only on the data, such that $|b| \leq L$.
\end{remark}

\begin{remark}
    It is worth commenting that, insofar as regularity of $u$ is concerned, the assumption \eqref{chochin}
    required in Lemma \ref{int_u}, is not restrictive. Indeed, if $u$ is a local minimizer of \eqref{functional_rho}, then arguing as in Lemma \ref{scaling}, a scaled version of $u$ will be a local minimizer of an energy functional as in \eqref{functional_rho}, with the norm of the bound in \eqref{bounded-perturbation} as small as needed: first fix an amplitude using the local $L^\infty$ bound, then shrink the spatial radius. The parameters used to scale the function will then only impact the constants of the H\"older estimate. See the details in the proof of the next lemma. 
\end{remark}

Next, we iterate the previous result to obtain the H\"older continuity of a local minimizer in a ball.

\begin{lemma}\label{HR0}
Let $u$ be under the hypotheses of Lemma \ref{int_u}. For any $\alpha$ satisfying \eqref{agniek}, there exist $\tau_0 \in \left( 0, \frac14 \right)$ and a constant $C>0$,  depending only on universal quantities and $\alpha$, such that if $u$ is a local minimizer of \eqref{functional_rho} satisfying \eqref{hypo} and \eqref{chochin},
then
\begin{equation}\label{LH_reg}
    \left| u(x)-u(0) \right| \leq C|x|^{\alpha}, 
\end{equation}
for every $x \in B_{\tau_0}$.
\end{lemma}

\begin{proof}
Given $\alpha$ satisfying \eqref{agniek}, we consider $\tau_0 \in \left( 0, \frac14 \right)$ the corresponding radius given by Lemma \ref{int_u}. 

For $u$ a local minimizer of \eqref{functional_rho} satisfying \eqref{chochin}, we claim that
\begin{equation}\label{ind_hyp}
    \dashint_{B_{\tau_0^k}} \left| u(x)-b_k \right|^2 dx \leq \tau_0^{2k\alpha},\qquad \text{for}\ k=1,2,\ldots,
\end{equation}
for some convergent sequence $(b_k)$. We will use induction over $k$ to prove the claim, the case $k=1$ being Lemma \ref{int_u}, for some constant $b_1$ such that $|b_1|\leq L$, with $L$ depending only on the data. Now, suppose that \eqref{ind_hyp} is true for $k>1$ and some constant $b_k$, and let's show that the same is true for $k+1$ and a constant $b_{k+1}$. We define
$$
    w_k(x)=\frac{u \left( \tau_0^k x \right)-b_k}{\tau_0^{k\alpha}}, \qquad x \in B_1,
$$
and observe that $w_k$ is a local minimizer of the functional
$$
    \int_{B_1} \left( \langle \tilde{A}(x)\nabla v , \nabla v  \rangle + \rho_k(x,v) \right) dx,
$$
where $\tilde{A}(x)={A} \left( \tau_0^k x \right)$ and $\rho_{k}(x,v)= \tau_0^{2k(1-\alpha)}\rho \left( \tau_0^k x, \tau_0^{k\alpha} v +b_k \right)$. Note that this functional remains in the same class. Indeed,
\begin{equation}\label{rhok}
 |\rho_k(x,v)|\le\tau_0^{2k(1-\alpha)}f(\tau_0^kx).
\end{equation}
Since $\alpha<1-n/(2q)$, the corresponding bound satisfies
\[
 \|\tau_0^{2k(1-\alpha)}f(\tau_0^k\,\cdot)\|_{L^q(B_1)}
 \le\tau_0^{k(2-2\alpha-n/q)}\|f\|_{L^q(B_1)}\le\delta_0.
\]
Thus \eqref{chochin} is preserved without any dependence on $b_k$.

Now, due to the induction hypothesis \eqref{ind_hyp}, we have
$$
    \dashint_{B_1} w_k^2 \, dx= \dashint_{B_1} \frac{\left| u \left( \tau_0^k x \right) -b_k \right|^2}{\tau_0^{2k\alpha}} \, dx \leq 1.
$$
Thus, $w_k$ satisfies the hypotheses of Lemma \ref{int_u}, and there exists a constant $b$, with $|b|\leq L$ such that
$$
    \dashint_{B_{\tau_0}} \left| w_k(x)-b \right|^2 dx \leq \tau_0^{2\alpha}. 
$$
Translating this back to $u$, we get
$$
    \dashint_{B_{\tau_0^{k+1}}} \left| u(x)-b_k-\tau_0^{k\alpha}b \right|^2 dx\leq \tau_0^{2(k+1)\alpha}.
$$
Thus, choosing $b_{k+1}=b_k+\tau_0^{k\alpha}b$, we obtain \eqref{ind_hyp} for $k+1$ as desired.

To show that the sequence $(b_k)$ is convergent, we will show that it is a Cauchy sequence. First, note that, by construction, 
$$
    \left| b_{k+1}-b_k \right| \leq L \tau_0^{k\alpha}.
$$
Then, for $b_m$ and $b_k$, with $m>k$, we have
\begin{eqnarray*}
|b_m-b_k| & \leq &  |b_{m}-b_{m-1}|+\ldots+|b_{k+1}-b_k|\\
& \leq & L\sum_{i=k}^{m-1}\tau_0^{i \alpha}\\
& =&  \frac{L \tau_0^{k\alpha}}{1-\tau_0^\alpha} \left[ 1-\tau_0^{(m-k)\alpha} \right]. 
\end{eqnarray*}
Since $0<\tau_0<1/4$, taking $k\rightarrow\infty$, we get
$$
    \left| b_m-b_k \right| \rightarrow 0
$$
and thus $(b_k)$ is a Cauchy sequence, so it converges to some $b_0$, with 
$$
    |b_k-b_0|\leq \frac{L \tau_0^{k\alpha}}{1-\tau_0^\alpha}.
$$

Finally, for $0<\tau<\tau_0$, we choose $k$ such that $\tau_0^{k+1}<\tau\leq\tau_0^k$. By \eqref{ind_hyp}, we get
\begin{eqnarray*}
    \dashint_{B_\tau} \left| u(x) -b_0 \right|^2 dx &\leq & 2 \dashint_{B_\tau} \left| u(x) -b_k \right|^2 dx + 2 \dashint_{B_\tau} \left| b_k-b_0 \right|^2 dx\\
    &\leq & 2 \tau_0^{-n} \dashint_{B_{\tau_0^k}} \left| u(x) -b_k \right|^2 dx + 2 \tau_0^{-n} \dashint_{B_{\tau_0^k}}\left| b_k-b_0 \right|^2 dx \\
    &\leq & 2 \tau_0^{-n} \left(1 + \frac{L^2}{(1-\tau_0^\alpha)^2}\right)\tau_0^{2k\alpha}\\
    &\leq & 2 \tau_0^{-n-2\alpha} \left(1 + \frac{L^2}{(1-\tau_0^\alpha)^2}\right) \tau^{2\alpha}.
\end{eqnarray*}
Repeat these mean-oscillation estimates at each $y\in B_{1/2}$ by applying the argument to $2^{-n/2}u(y+x/2)$. Its mean square on $B_1$ is at most $1$, and the norm of its potential bound is at most $2^{-n-2+n/q}\delta_0\le\delta_0$. The estimates are therefore uniform in $y$. Campanato's characterization of H\"older continuity gives a continuous representative and \eqref{LH_reg}.
\end{proof}

The H\"older regularity for local minimizers of \eqref{functional_rho} follows readily from the previous lemma, and, in particular, we obtain the first part of Theorem \ref{thm-main}.

\begin{proof}[Proof of Theorem \ref{thm-main}--(1)]
By \eqref{original-bound}, $K=\|u\|_{L^\infty(B_{3/4})}\le C\|u\|_{L^2(B_1)}$.
Truncate the potential in \eqref{funct} outside $[0,K]$. Truncating competitors shows that $u$ remains a local minimizer, and the new potential is bounded by
\[
 F=\Lambda\max\{1,K^{\gamma^\star}\}.
\]
Fix $0<\epsilon<\alpha_0$ and the amplitude $\sigma=\max\{1,K,\sqrt F\}$.
For $x_0\in B_{1/2}$ choose a fixed $0<r<1/4$ with $r^2\le\delta_0$, where $\delta_0$ is from Lemma \ref{int_u} with $q=\infty$ and $\alpha=\epsilon$. Then
\[
 v(x)=\sigma^{-1}u(x_0+rx),\qquad x\in B_1,
\]
has mean square at most $1$ and minimizes an energy whose potential has absolute value at most $r^2F/\sigma^2\le\delta_0$. Lemma \ref{HR0} therefore gives
\[
 |u(x)-u(x_0)|\le C\sigma |x-x_0|^\epsilon,
 \qquad x\in B_{r\tau_0}(x_0).
\]
Together with the bound by $K$ for pairs farther apart, this proves
\[
 \|u\|_{C^{0,\epsilon}(B_{1/2})}
 \le C\left(1+\|u\|_{L^2(B_1)}+\|u\|_{L^2(B_1)}^{\gamma^\star/2}\right).
\]
Here $C$ has the dependencies stated in the theorem; we may fix $\epsilon=\alpha_0/2$.
If $0\le u\le1$ and $\Lambda\le1$, then $\sigma=1$ and these H\"older bounds depend only on $n,\mu$.
\end{proof}

\section{Regularity gain} \label{sct reg gain}

In this section, we prove the crucial part of Theorem \ref{thm-main}, which provides an accurate pointwise regularity at free boundary points for local minimizers with Dini-continuous oscillatory singularities. Thus, hereafter, we will focus on non-negative local minimizers of the original energy functional \eqref{funct}. For simplicity, we restrict the analysis to the case $\Omega=B_1$. We start with a quantitative flatness improvement result.  

\begin{lemma} \label{florida-mall}
Given any $0<\varrho \le 1$, there exists $ \lambda_\varrho >0$, depending only on $n$, $\mu$ and $\varrho$, such that if $u$ is a local minimizer of \eqref{funct} satisfying $0 \leq u \leq 1$ in $B_1$ and $u(0)=0$, with
$$
    0\le \lambda(x) \leq \lambda_\varrho,
$$
then 
$$
    \sup_{B_{1/2}} u \leq \varrho.
$$
\end{lemma}
\begin{proof} We take $\lambda_\varrho\le1$. Let $h$ be the $A$-harmonic replacement of $u$ in $B_{3/4}$. Revisiting the proof of Lemma \ref{Approx_lemma}, with a closer look at estimate \eqref{use later}, using $0\le u,h\le1$, we reach
\begin{equation}\label{flat eq1}
    \dashint_{B_{3/4}} \left| u-h \right|^2 dx \le C_0 \|\lambda\|_\infty,
\end{equation}
where $C_0$ depends only on dimension and ellipticity. Next we note that, since $u$ is a local minimizer of \eqref{funct}, it satisfies
\begin{equation}\label{flat eq2}
    \text{div} (A(x) \nabla u) \ge 0,
\end{equation}
in the distributional sense, \textit{i.e.}, $u$ is $A$-subharmonic. In particular, $0\le u \le h$ in $B_{3/4}$. Applying the Harnack inequality for $h$, along with standard $L^2$--$L^\infty$ estimates for $h$, we obtain
\begin{equation}\label{flat eq3}
    \sup\limits_{B_{1/2}} u \le \sup\limits_{B_{1/2}} h \le C_1h(0) \le C_1C_2r^{-n/2}\|h\|_{L^2(B_r)},
\end{equation}
where $C_1$ and $C_2$ are universal constants and $0<r<1/2$ is arbitrary. Next, using the already proven $\epsilon-$H\"older continuity of $u$ at $0$, together with the fact that $u(0) = 0$, we estimate
\begin{equation}\label{flat eq4}
    \begin{array}{lll}
        \displaystyle \dashint_{B_{r}} h^2 \, dx & \le & 2 \left ( \displaystyle \dashint_{B_{r}} \left| h-u \right|^2 dx +  \dashint_{B_{r}} u^2 \, dx \right ) \\
        &\le& 2 \left ( \displaystyle \left(\frac{3}{4r}\right)^n \dashint_{B_{3/4}} \left| h-u \right|^2 dx +  C r^{2\epsilon} \right ). 
    \end{array}
\end{equation}
Combining \eqref{flat eq1}, \eqref{flat eq3}, and \eqref{flat eq4} we reach
\begin{eqnarray}
    \left (\sup\limits_{B_{1/2}} u \right )^2 &\le& \displaystyle (C_1C_2)^2 \omega_n \dashint_{B_{r}} h^2 \, dx \nonumber \\
    &\le & 2 (C_1C_2)^2 \omega_n \left( \displaystyle \left(\frac{3}{4r}\right)^n C_0 \|\lambda\|_\infty +  C r^{2\epsilon} \right), \label{flat eq5}
\end{eqnarray}
where $\omega_n$ is the volume of the unit ball in $\R^n$. Now, given $0<\varrho \le 1$, we select $0<r_\varrho \ll 1/2$ satisfying
\begin{equation}\label{flat eq6}
     2 (C_1C_2)^2 \omega_n C r_\varrho^{2\epsilon} = \frac{1}{2} \varrho^2.
\end{equation}
Finally, with $0<r_\varrho \ll 1/2$ chosen, we take $\lambda_\varrho >0$ so that
\begin{equation}\label{flat eq7}
     2 (C_1C_2)^2 \omega_n \left(\frac{3}{4r_\varrho}\right)^n C_0\lambda_\varrho \le \frac{1}{2} \varrho^2,
\end{equation}
and the proof of the lemma is complete.
\end{proof}

\medskip

Before proving our main result, we analyze the assumptions on $\lambda$ and $\gamma$ that guarantee the scaling invariance for minimizers of the functional \eqref{funct}. Given parameters $0<r\leq 1$ and $\beta \geq 1$, let
$$
    v(x) := \frac{u(rx)}{r^{\beta}}, \quad x \in B_1.
$$
Recalling Lemma \ref{scaling}, if $u$ is a local minimizer of \eqref{funct} in $B_r$ then $v$ is a local minimizer of the functional
$$
    \int_{B_1} \left( \langle A_r(x)\nabla v , \nabla v \rangle +\lambda_r(x)\, v^{\gamma_r(x)} \chi_{\{ v >0 \} } \right) dx,
$$
with
\begin{equation}\label{scalingA}
    A_r(x):=A(rx), \qquad \gamma_r(x):=\gamma(rx) 
\end{equation}
and 
\begin{equation}\label{scalingB}
    \lambda_r(x):=\lambda(rx)r^{\gamma(rx)\beta-2(\beta-1)}=\lambda(rx)r^{\beta\left( \gamma(rx)-2+\frac{2}{\beta}\right)}.
\end{equation}
Observe that $A_r$ satisfies the same ellipticity condition as $A$ and $\gamma_r$ satisfies the same bounds as $\gamma$. To guarantee that $\lambda_r$ is universally bounded, we assume that
$$
    \gamma_\star(r)\ge 2-\frac{2}{\beta} \quad \Longrightarrow \quad \gamma(rx)\geq 2 - \frac{2}{\beta}, \ x\in B_1,
$$
where 
$$
    \gamma_\star(r):= \inf_{x \in B_r} \gamma(x).
$$

\medskip

We now prove the main result in this paper, namely the sharp higher regularity at free boundary points.
	
\begin{proof}[Proof of Theorem \ref{thm-main}--(2)] We do the analysis at the origin $x_0=0$ and divide the proof into three steps. Recall we are assuming $\gamma (0) <2$ and that $\gamma$ is Dini continuous at $0$.

\medskip

\noindent\textbf{Step 1} (Pointwise scaling invariance). Let $M\geq 1$ be such that
$$
    \log_2 M =  \frac{2}{2-\gamma (0)} \displaystyle \sum_{k=1}^\infty \omega \left( 2^{-k} \right),
$$
where $\omega$ is the modulus of continuity of $\gamma$ at $0$. Since  $\gamma$ is Dini continuous, the above series converges, and $M$ is well defined. 

For $\beta=\frac{2}{2-\gamma (0)}$, take $r=2^{-k}$ in \eqref{scalingB}, for $k \in \N$, obtaining
$$
    \lambda_k(x):=\lambda_{r}(x)=\lambda \left( 2^{-k}x \right)2^{2k \frac{\gamma(0)-\gamma \left( 2^{-k}x \right)}{2-\gamma (0)}}.
$$
If $\gamma(0) \leq \gamma \left( 2^{-k}x \right)$, we easily get $\lambda_k(x) \leq \lambda \left( 2^{-k}x \right)$. Otherwise, 
\begin{eqnarray*}
\lambda_k(x) & \leq &  \lambda \left( 2^{-k}x \right) 2^{\frac{2}{2-\gamma (0)} k \omega \left( 2^{-k} \right)} \\
& \leq &  \lambda \left( 2^{-k}x \right) 2^{\frac{2}{2-\gamma (0)} \sum_{i=1}^k \omega \left( 2^{-i} \right)} \\
& \leq &  \lambda \left( 2^{-k}x \right) 2^{\frac{2}{2-\gamma (0)} \sum_{k=1}^\infty \omega \left( 2^{-k} \right)} \\
& = &\lambda \left( 2^{-k}x \right)  M,
\end{eqnarray*}
where the second inequality follows from the fact that $\omega(\cdot)$ is increasing because it is a modulus of continuity. Therefore, 
\begin{equation}\label{scalingC}
    \lambda_k(x) \leq (1+M) \, \lambda \left( 2^{-k}x \right).
\end{equation}

\medskip

\noindent \textbf{Step 2} (Pointwise oscillation estimates in dyadic balls). We now assume 
\begin{equation}\label{norm}
    0 \leq u \leq 1 \quad \mbox{in} \  B_1 .
\end{equation}
Take
\begin{equation}\label{Tivoli}
    \varrho := 2^{\frac{2}{\gamma (0)-2}}
\end{equation}
and let $\lambda_\varrho >0$ be the corresponding number given by Lemma \ref{florida-mall}. We also assume
\begin{equation}\label{scalingD}
	\|\lambda \|_\infty \leq \frac{1}{1+M} \, \lambda_\varrho.
\end{equation}

Now, for each $k \in \N$, we claim that 
\begin{equation}\label{dyadic}
    \sup_{B_{2^{-k}}} u \leq  2^{\frac{2k}{\gamma (0)-2}}.
\end{equation}
We proceed by induction, the case $k=1$ being given by Lemma \ref{florida-mall}, with $\varrho$ as in \eqref{Tivoli}. Let's assume \eqref{dyadic} holds for $k$ and show the same is true for $k+1$. Define
$$
    v_k(x) := 2^{\frac{2k}{2-\gamma (0)}} u \left( 2^{-k}x \right), \quad x \in B_1.
$$
Observe that $0 \leq v_k \leq 1$ in $B_1$, by the induction hypothesis, and that $v_k(0)=0$ because $0 \in \partial\{u>0\}$. In addition, $v_k$ is a minimizer of \eqref{funct}, for $A=A_r$, $\gamma=\gamma_r$ and $\lambda=\lambda_r$, as defined in \eqref{scalingA} and \eqref{scalingB}, for $r=2^{-k}$ and $\beta=\frac{2}{2-\gamma (0)}$. Thus, from \eqref{scalingC} and \eqref{scalingD}, we have 
$$
    \lambda_r (x) = \lambda_k (x) \leq \lambda_\varrho,
$$
and so $v_k$ satisfies the hypotheses of Lemma \ref{florida-mall} and we conclude 
$$
    \sup_{B_{1/2}} v_k \leq \varrho = 2^{\frac{2}{\gamma (0)-2}},
$$
which, translated into $u$, gives
$$
    \sup_{B_{2^{-(k+1)}}} u \leq 2^{\frac{2(k+1)}{\gamma (0)-2}}.
$$
The induction is complete.

\medskip

\noindent \textbf{Step 3} (Smallness regime). Since $\gamma(x_0) < 2$, fix $0<r_0<1/4$ such that 
$$
    \gamma_0^\star := \sup_{B_{r_0} (x_0)} \gamma <2.
$$
We now need to adjust the parameter $\kappa$ to guarantee that
$$
    w(x):= \kappa^{-1} u(x_0+r_0x), \quad x \in B_1,
$$
satisfies conditions \eqref{norm} and \eqref{scalingD}. Choosing 
$$
    \kappa:= \max\left\{\|u\|_{L^\infty(B_{3/4})},
     \left( \frac{\| \lambda\|_\infty (1+M)}{\lambda_\varrho} \right)^{\frac{1}{2-\gamma_0^\star}} ,  \;1\right\},
$$
we observe that $w$ is a minimizer of \eqref{funct} with coefficient $\lambda$ satisfying \eqref{scalingD} and that $\|w\|_\infty \leq 1$. 

Applying Step 2 to $w$ and interpolating between dyadic radii, with $\beta=2/(2-\gamma(x_0))$, gives
\[
 \sup_{B_r(x_0)}u\le C_0r^\beta,\qquad
 C_0=2^\beta\kappa r_0^{-\beta},\qquad 0<r\le r_0/2.
\]
For $r_0/2\le r\le1/4$, the same bound follows from $u\le\kappa$ in $B_{3/4}$. This proves the stated estimate and its dependencies. 
\end{proof}

\begin{remark}
The regularity gain established in this section and captured by estimate \eqref{sharp growth} is independent of the lower bound $c_0>0$ imposed on the forcing term $\lambda (x)$. The sole purpose of this lower bound is to guarantee the sharpness of \eqref{sharp growth}. While this observation has further consequences within the broader free boundary program (see \cite{ASTU1}), we will not explore those implications in this paper.
\end{remark}

\section{Free boundary repelling estimate} \label{sct repel}

In this section, we present the third and concluding piece of information incorporated into the statement of Theorem \ref{thm-main}. The goal is to show that if $\gamma(x_0) > 2$, then $x_0$ must be at a quantitatively controlled positive distance from the free boundary. 

\begin{proof}[Proof of Theorem \ref{thm-main}--(3)] The proof involves three steps. The key step is to show that $x_0$ cannot be a free boundary point. This will be attained by showing that if $\gamma(x_0) > 2$, then $\sup_{B_r} u$ cannot be strictly positive for all $r>0$. The second step is to verify that if $x_0 \in \{u=0\} \cap \gamma^{-1}(2, 2^*)$, then $B_{\varrho_0}(x_0) \subset \{u=0\}$. This will be done by a contradiction argument, reducing the problem to the free boundary, already discussed in Step 1. The third and final step is to show that if $x_0 \in \{u > 0 \} \cap \gamma^{-1}(2, 2^*)$, then $\text{dist}(x_0, \partial \{u>0 \} ) \ge \varrho_0 >0$. This will follow from simple continuity considerations. Throughout the proof, we will denote by $\omega$ the (fixed) modulus of continuity of $\gamma(x)$. 

\medskip 

\noindent {\bf Step 1} (The free boundary cannot intersect $\gamma^{-1}(2,2^*)$). Let us assume $x_0$ is a free boundary point corresponding to a local minimizer of \eqref{funct} with $\gamma(x_0) > 2$. By continuity, we can assume
$$
    \inf\limits_{B_{r_0}(x_0)} \gamma \ge 2,
$$
for some $r_0>0$, with $\overline{B_{r_0}(x_0)}\subset B_1$, depending only on $\gamma(x_0)$, $\omega$ and $\dist(x_0,\partial B_1)$. We claim that for every $0<\delta<1$ there exists a natural number $j_0=j_0(\delta)>-\log_2 r_0$ such that
\begin{equation}\label{magic}
    \sup\limits_{B_{2^{-(j+1)}}(x_0)}u\le\delta\sup\limits_{B_{2^{-j}}(x_0)}u,\qquad j\ge j_0. 
\end{equation}
The index $j_0$ is allowed to depend on $\delta$.

To prove the claim, suppose its thesis fails to hold. This implies the existence of a positive $\delta_0 >0$ and a sequence $j_k \to \infty$ such that
$$
     \sup\limits_{B_{2^{-(j_k+1)}}(x_0)} u  >  \delta_0 \sup\limits_{B_{2^{-j_k}}(x_0)} u, 
$$
for all $k \in \mathbb{N}$.  Hereafter, for each $0<r<r_0$, we will denote
$$
    S_r := \sup\limits_{B_r(x_0)} u.
$$
Since $x_0\in\partial\{u>0\}$ and $u$ is continuous, $0<S_r\le1$ for all sufficiently small $r>0$; we discard finitely many terms of the sequence if necessary. Next, we define
$$
    v_k(x) := \frac{u(x_0 + 2^{-j_k}x)}{S_{2^{-(j_k+1)}}}.
$$
We easily check that $v_k\colon B_1 \to \mathbb{R}$ satisfies the following four assertions.

\begin{itemize}
    
    \item $0\le v_k(x) \le \delta_0^{-1}$, for all $x\in B_1$;

    \medskip

    \item $v_k(0) = 0$;

    \medskip
    
    \item $\sup\limits_{B_{1/2}} v_k = 1.$

    \medskip
    
    \item $v_k$ is a local minimizer of
    $$
        \mathscr{F} (v) = \int \langle A_k(x) \nabla v, \nabla v \rangle + \lambda_k (x) v^{\gamma_k (x)} \chi_{\{v> 0 \}} dx,
    $$
    where $A_k(x) = A(x_0 + 2^{-j_k}x)$, $\gamma_k (x) = \gamma (x_0 + 2^{-j_k}x)$ and, more importantly,
    $$
        \lambda_k(x) = \left [ S_{2^{-(j_k+1)}} \right ]^{\gamma_k(x) - 2} 4^{-j_k} \lambda(x_0 + 2^{-j_k}x).
    $$
\end{itemize}
The key remark here is that, since $\gamma_k(x) - 2 \ge 0$, $\| \lambda_k\|_\infty = \text{o}(1)$, as $k\to \infty$.

\medskip

We now revisit the analysis carried out in the previous section. For each $j_k$, let $h_k$ be the $A_k$-harmonic replacement of $v_k$ in $B_1$. As in the proof of Lemma \ref{florida-mall}, we obtain
\begin{equation}\label{C-Eq2}
    \dashint_{B_1} \left| \nabla (v_k - h_k)\right|^2 dx \le C \| \lambda_k\|_\infty = \text{o}(1),
\end{equation}
as $k\to \infty$. Here $C$ may depend on $\delta_0$ and $\gamma^\star$. This is because, as commented above, by assumption, $\gamma_k(x) \ge 2$ for all $k$ and, as also previously noted, $ 0< S_{2^{-(j_k+1)}} \le 1$. Applying the maximum principle and Harnack's inequality for $h_k$, we reach
\begin{equation}\label{C-Eq3}
   \sup\limits_{B_{2/3}} v_k \le \sup\limits_{B_{2/3}} h_k \le C_1 h_k(0),
\end{equation}
for a universal constant $C_1>0$. Also, $0\leq h_k \le \delta_0^{-1}$. In view of the $C^{0,\epsilon}$ regularity estimate proven for $v_k$ and interior Caccioppoli estimates, there exists a function $v_\infty$, such that, up to a subsequence, $v_k \to v_\infty$ locally uniformly in $B_1$ and weakly in $H^1_{\mathrm{loc}}(B_1)$. Also, by known estimates for $h_k$, we can assume $h_k$ converges locally uniformly in $B_1$ and weakly in $H^1_{\mathrm{loc}}(B_1)$ to a function $h_\infty$. Due to \eqref{C-Eq2} and Poincar\'e's inequality, we conclude
$$
    v_\infty \equiv h_\infty,
$$
in $B_1$. In particular, because of uniform convergence, 
$$
    v_k(0) = v_\infty(0) = h_\infty(0) = \lim\limits_{k\to \infty}  h_k(0) = 0.
$$
Thus, letting $k\to \infty$ in \eqref{C-Eq3}, we conclude
$$
    v_\infty \equiv 0 \quad \text{ in } B_{2/3}.
$$
However, this leads to a contradiction as, per our construction, we have $\sup_{B_{1/2}} v_k = 1$ for all $k\ge1$.

To conclude Step 1, we need a lower bound for the ratios in \eqref{magic}. Since $\gamma\ge2$ on $B_{r_0}(x_0)$, the potential is differentiable at zero and the Euler--Lagrange equation holds weakly throughout this ball:
\[
 -\operatorname{div}(A\nabla u)+c(x)u=0,\qquad
 c(x)=\begin{cases}
 \dfrac{\lambda(x)\gamma(x)}{2}u^{\gamma(x)-2},&u>0,\\[4pt]
 0,&u=0.
 \end{cases}
\]
Local boundedness of $u$ gives $0\le c\in L^\infty(B_{r_0}(x_0))$.
Harnack's inequality \cite[Theorem 8.20]{GT} therefore yields
\[
 S_r\le C_H S_{r/2},\qquad 0<r<r_0/2,
\]
with $C_H$ independent of $r$. Fix $0<\delta<C_H^{-1}$ in \eqref{magic} and enlarge $j_0(\delta)$ so that $2^{-j_0}<r_0/2$. Combining these inequalities forces $S_{2^{-j_0}}=0$, and hence
\begin{equation}\label{C-Eq4}
    \sup\limits_{B_{\varsigma_0}(x_0)} u = 0,
\end{equation}
where $\varsigma_0 = 2^{-(j_0 + 1)}$. This ultimately gives a contradiction, as $x_0$ was assumed to be a free boundary point, \textit{i.e.}, $x_0 \in \partial \{u> 0 \}$, whereas equation \eqref{C-Eq4} clearly forces $x_0$ to be in the interior of the contact zero set $\{u = 0\}$. 

\bigskip

\noindent {\bf Step 2} (From the zero set to the free boundary). Assume $x_0 \in \{u=0\}$, and seeking a contradiction, that no such $\varrho_0>0$ exists. This means there exists a sequence of uniformly bounded, non-negative local minimizers $\{u_n\}$ of an energy functional
$$
     \mathscr{F}_n(v) = \int \langle A(x) \nabla v, \nabla v \rangle + \lambda_n(x) v^{\gamma_n(x)} \chi_{\{v> 0 \}} dx,
$$
with 
\begin{itemize}
    
    \item $x_0 \in Z(u_n)$;
    
    \medskip
    
    \item $\text{dist}(x_0, \partial \{u_n> 0 \}) <1/n$;
    
    \medskip
    
    \item $\|\lambda_n\|_\infty \le \Lambda$;
    
    \medskip
    
    \item $\gamma_n(x_0) \ge \gamma(x_0) > 2$;
    
    \medskip
    
    \item $\gamma_n$ has the same modulus of continuity, $\omega$, as $\gamma$;
    
    \medskip
    
    \item $0\le \gamma_n \le \gamma^\star < 2^*$.  
    
    \end{itemize}

\medskip

For all sufficiently large $n$, choose $z_n\in\partial\{u_n>0\}\cap B_{3/4}$ with $|z_n-x_0|<1/n$. Step 1 applies at $z_n$ and gives $\gamma_n(z_n)\le2$. Therefore
\[
\begin{aligned}
 0<\gamma(x_0)-2
 &\le\gamma_n(x_0)-\gamma_n(z_n)\\
 &\le\omega(|z_n-x_0|)\le\omega(1/n)\longrightarrow0,
\end{aligned}
\]
a contradiction. More precisely, choose $0<\varrho_0<1/4$ with $\omega(2\varrho_0)<\gamma(x_0)-2$. The same argument excludes every free boundary point in $B_{2\varrho_0}(x_0)$. Since $u(x_0)=0$, continuity and connectedness imply $u\equiv0$ on this ball. This proves the assertion for the zero set, with dependence on $\gamma(x_0)-2$ and $\omega$.

\bigskip

\noindent {\bf Step 3} (Estimate in the positive set). Assume $x_0 \in \{u>0 \} \cap B_{1/2}$ is such that $\gamma(x_0) = 2 + \nu$, for some $0<\nu< 2^* - 2$. Choose $0<\varrho_0<1/4$ with $\omega(2\varrho_0)<\nu$. If the distance to the free boundary is at least $\varrho_0$, there is nothing to prove. Otherwise, let $z_0\in\partial\{u>0\}\cap B_{3/4}$ be a free boundary point satisfying
$$
    |z_0 - x_0| = \text{dist}(x_0, \partial\{u>0 \}) =:d.
$$
By the conclusion of Step 1, we know $\gamma(z_0) \leq 2$, and then
$$
     \gamma(x_0) - \gamma(z_0) \geq \nu.
$$
On the other hand, by continuity, we can estimate
$$
     \gamma(x_0) - \gamma(z_0) \le \omega(|z_0 - x_0|) = \omega(d).
$$
The assumption $d<\varrho_0$ would imply $\nu\le\omega(d)<\nu$, a contradiction. Hence, we finally obtain
$$
   d\ge\varrho_0,
$$
and the proof is complete.
\end{proof}

\bigskip

{\small \noindent{\bf Acknowledgments.} This publication is based upon work supported by King Abdullah University of Science and Technology (KAUST) under Award No. ORFS-CRG12-2024-6430. DJA is partially supported by CNPq grants 420014/2023-3 and 310020/2022-0. GSS was supported by a CNPq postdoctoral fellowship. ET thanks the partial support offered by Grayce B. Kerr funds at Oklahoma State University.}

\bigskip

\bibliographystyle{amsplain, amsalpha}

\end{document}